\documentclass{amsart}
\usepackage{amsmath,amsthm,amssymb}
\usepackage{graphicx}

\newtheorem{theorem}{Theorem}[section]
\newtheorem{proposition}[theorem]{Proposition}
\newtheorem{lemma}[theorem]{Lemma}

\newtheorem{conjecture}[theorem]{Conjecture}

\theoremstyle{definition}

\newtheorem{remark}[theorem]{Remark}

\newtheorem{question}[theorem]{Question}

\newtheorem{step}{Step}

\theoremstyle{theorem}

\title[Maximal Thurston-Bennequin number and reducible surgery]{Maximal Thurston-Bennequin number and reducible Legendrian surgery}
 
\author[Kouichi Yasui]{Kouichi Yasui}
\thanks{The author was partially supported by JSPS KAKENHI Grant Number 25800048.}
\date{September 10, 2015. Revised: February 3, 2016}

\address{Department~of~Mathematics, Graduate School~of~Science, Hiroshima~University, 1-3-1 Kagamiyama, Higashi-Hiroshima, 739-8526, Japan}
\email{kyasui@hiroshima-u.ac.jp}
\begin{document}

\begin{abstract}We give a method for constructing a Legendrian representative of a knot in $S^3$ which realizes its maximal Thurston-Bennequin number under a certain condition. The method utilizes Stein handle decompositions of $D^4$, and the resulting Legendrian representative is often very complicated (relative to the complexity of the topological knot type). As an application, we construct infinitely many knots in 
$S^3$ each of which yields a reducible 3-manifold by a Legendrian surgery in the standard tight contact structure. This disproves a conjecture of Lidman and Sivek. 
\end{abstract}

\maketitle

\section{Introduction}\label{sec:intro}The maximal Thurston-Bennequin number, denoted by $\overline{tb}$, of a knot in $S^3$ is an important invariant in low dimensional topology. For example, a 4-manifold represented by a knot with framing less than $\overline{tb}$ admits a Stein structure (\cite{E1}, cf.\ \cite{G1}), and such Stein 4-manifolds have many applications to low dimensional topology, e.g., exotic smooth structures, contact 3-manifolds and 4-genera of knots (cf.\ \cite{GS}, \cite{OS1}). Here recall that $\overline{tb}(K)$ of a knot $K$ in $S^3$ is the maximal value of the Thurston-Bennequin number $tb(\mathcal{K})$ of a Legendrian representative $\mathcal{K}$ of $K$ in the standard tight contact structure on $S^3$. 
There are several invariants which give good upper bounds for $\overline{tb}$ of knots (e.g.\ \cite{AM}, \cite{LM2}, \cite{FT}, \cite{OzSz_tau}, \cite{Pl}, \cite{Pl_Kho}, \cite{Ras}, \cite{Shu}, \cite{Ng1}, \cite{Ng2}). By contrast, it is generally difficult to find a Legendrian representative of a knot realizing an upper bound of $\overline{tb}$ when the crossing number of the knot is large. Hence determining $\overline{tb}$ is a difficult problem in general. 

In this paper, we give a method for constructing a Legendrian representative of a knot in $S^3$ which realizes its maximal Thurston-Bennequin number under a certain condition. 
The method utilizes Stein handle decompositions of $D^4$, and the resulting Legendrian representative (in the front diagram of $S^3$) is often very complicated (relative to the complexity of the topological knot type). 
One can easily construct various examples of knots for which this method effectively works, and it seems difficult to determine their $\overline{tb}$ by other methods.

As an application of our method, we discuss reducible Legendrian surgeries. A long standing open problem in Dehn surgery theory is to determine framed knots in $S^3$ which produce reducible 3-manifolds. The cabling conjecture \cite{GAS} 
 asserts a complete characterization of such framed knots, and there are many related studies (see \cite{Gr} and the references therein). 
 Recently Lidman and Sivek~\cite{LSi} gave an interesting new approach to this problem from contact topology. Here we recall basic facts. The Legendrian surgery along a Legendrian knot $\mathcal{K}$ in the standard tight contact structure on $S^3$ is topologically the Dehn surgery along $\mathcal{K}$ with the contact $-1$ (i.e.\ $tb(\mathcal{K})-1$) framing, and any integral Dehn surgery along a knot $K$ in $S^3$ with framing less than $\overline{tb}(K)$ can be realized as a Legendrian surgery along a Legendrian representative of $K$ in the standard contact structure. Applying Eliashberg's theorem (\cite{E2}) on splittings of Stein 4-manifolds with reducible boundary 3-manifolds, Lidman and Sivek proved the following theorem. 
\begin{theorem}[Lidman and Sivek~\cite{LSi}]For a knot $K$ in $S^3$ with $\overline{tb}(K)\geq 0$, any Dehn surgery along $K$ with coefficient less than $\overline{tb}(K)$ is irreducible.
\end{theorem}
In other words, any Legendrian surgery along a knot ${K}$ with $\overline{tb}({K})\geq 0$ in the standard contact structure on $S^3$ yields an irreducible 3-manifold. Moreover, they conjectured that this result holds without the assumption $\overline{tb}({K})\geq 0$. 
\begin{conjecture}[Lidman and Sivek~\cite{LSi}]\label{conjecture} 
A knot in $S^3$ never yields a reducible 3-manifold by a Legendrian surgery in the standard tight contact structure. 
\end{conjecture}
This conjecture has the following supporting evidence. The cabling conjecture asserts that a framed knot yielding a reducible 3-manifold is the $(p,q)$-cable of a knot with $pq$-framing, and the standard cabling construction (cf.\ Section 5 in \cite{Gol}) only gives a Legendrian representative of the cable knot with $tb\leq pq$. 

Here we disprove this conjecture applying the aforementioned method.  
\begin{theorem}\label{intro:thm:reducible}There exist infinitely many knots in $S^3$ each of which yields a reducible 3-manifold by a Legendrian surgery in the standard tight contact structure. Furthermore, each knot $K$ can be chosen so that the surgery coefficient is arbitrarily less than $\overline{tb}(K)$. 
\end{theorem}

In fact, we give a general method for constructing counterexamples. As an example, we will discuss the $(n,-1)$ cable $K_{m,n}$ of the ribbon knot $K_m$ in Figure~\ref{fig:ribbon_knot} for $n\geq 2$ and $m\leq -4n+3$. Although the standard cabling construction merely gives an estimate $\overline{tb}(K_{m,n})\geq -2n+1$ (see Figure~\ref{fig:cable} for a representative realizing this estimate), our method determines the explicit value $\overline{tb}(K_{m,n})=-1$ (Proposition~\ref{prop:tb}), implying the above theorem. Indeed our method yields the very complicated representative of $K_{m,n}$ in Figure~\ref{fig:knot_complicated} which realizes $\overline{tb}(K_{m,n})$. Here the notation and tangles $A_n$ and $B_n$ in the diagram are given in Figures~\ref{fig:notation} and \ref{fig:ABn}. 
We hope our method is useful for finding a new phenomenon in contact and symplectic topology. 

\subsection*{Acknowledgements}The author would like to thank Robert Gompf, Joshua Sabloff and Steven Sivek for their helpful comments on an earlier version of this paper. He also thanks the referee for his/her helpful comments. 
 
\section{Stein handlebody and notation} 
In this section, we recall basic definitions and properties. We also introduce our notations, some of which are different from the standard ones. 
\subsection{Stein handlebody}We briefly review basics of Stein handlebodies. For details, see \cite{G1} and \cite{GS}. For basics of contact topology and Legendrian knots, the readers can consult \cite{OS1}. 
Throughout this paper, we assume that a handlebody is 4-dimensional, compact, connected and oriented. 

Recall that a \textit{1-handlebody} (resp.\ \textit{2-handlebody}) is a handlebody which consists of $0$- and 1-handles (resp.\ 0-, 1- and 2-handles). 
We call a handlebody a \textit{Stein handlebody}, if it is constructed from a 1-handlebody $\natural_n S^1 \times D^3$ $(n\geq 0)$ by attaching 2-handles along a Legendrian link in the Stein fillable contact structure on the boundary $\#_n S^1 \times S^2$ such that the framing of each Legendrian knot is $-1$ relative to the framing induced from the contact plane (i.e.\ contact $-1$ framing). According to a result of Eliashberg \cite{E1} (cf.\ \cite{G1}), any Stein handlebody admits a Stein structure, extending the Stein structure on the 0-handle $D^4$. We note that each $\#_n S^1 \times S^2$ $(n\geq 0)$ admits a unique Stein fillable contact structure up to isotopy (\cite{E3}).  In the rest of this paper, a Legendrian link in $\#_n S^1 \times S^2$ $(n\geq 0)$ means the one with respect to the Stein fillable contact structure. By a result of Gompf \cite{G1} (cf.\ \cite{GS}), one can draw a Legendrian link in the boundary $\#_n S^1 \times S^2$ of a 1-handlebody using Gompf's  \textit{Legendrian link diagram in standard form}. In particular, we can draw a handlebody diagram of a Stein handlebody. 

The \textit{Thurston-Bennequin number} $tb(\mathcal{K})$ of a Legendrian knot $\mathcal{K}$ in the boundary $\#_n S^1 \times S^2$ of a 1-handlebody is defined to be the difference between the contact framing and the $0$-framing. Here recall that the \textit{$0$-framing} of a knot in $S^3$ (i.e.\ the boundary of 0-handle) is defined to be the Seifert framing (i.e.\ the one induced from a Seifert surface), and recall that the \textit{$0$-framing} of a knot in the boundary of a 1-handlebody is defined to be the Seifert framing induced from the dotted circle notation of the 1-handlebody. Consequently, if a knot bounds a Seifert surface in $\#_n S^1 \times S^2$, then the $0$-framing coincides with the framing induced from the surface. Note that a knot in $\#_n S^1 \times S^2$ bounds a surface if and only if the knot is null-homologous. For details of the Gompf's standard form diagram and calculation of the Thurston-Bennequin number in $\#_n S^1 \times S^2$, we refer to \cite{G1} and \cite{GS}. For the definition of $0$-framings, see \cite{GS}. 

Let $\mathcal{K}$ be a Legendrian knot in $S^3$, and let $g_4(\mathcal{K})$ denote the 4-genus of $\mathcal{K}$ (i.e.\ the minimal genus of a smoothly embedded surface in $D^4$ which bounds $\mathcal{K}$). One can estimate $g_4(\mathcal{K})$ as follows. By attaching a 2-handle to $D^4$ along $\mathcal{K}$ with framing $tb(\mathcal{K})-1$, we have a Stein 4-manifold. Since any Stein 4-manifold can be embedded into a closed minimal complex surface of general type with $b_2^+>1$ (\cite{LM1}), applying the adjunction inequality (\cite{FS3, KM1,MST, OzSz_ad}) for this closed 4-manifold  together with Gompf's Chern class formula (\cite{G1}), we obtain the following adjunction inequality, where $r(\mathcal{K})$ denotes the rotation number of $\mathcal{K}$. 
\begin{theorem}[\cite{LM1}, \cite{AM}, \cite{LM2}] $tb(\mathcal{K})+|r(\mathcal{K})|\leq 2g_4(\mathcal{K})-1$.\end{theorem}
Note that this holds even for the genus zero case (cf.\ \cite{GS, OS1}), unlike the version for general closed 4-manifolds. 

\subsection{Notations and definitions}Here we introduce our notations and definitions. Beware that some of our definitions are different from the standard ones. 

The 4-manifold represented by a framed knot in $S^3$ means the 4-manifold obtained from $D^4$ by attaching a 2-handle along the framed knot. 

For a Legendrian knot diagram, left and right handed twists are abbreviated as shown in Figure~\ref{fig:notation}. 

\begin{figure}[h!]
\begin{center}
\includegraphics[width=4.5in]{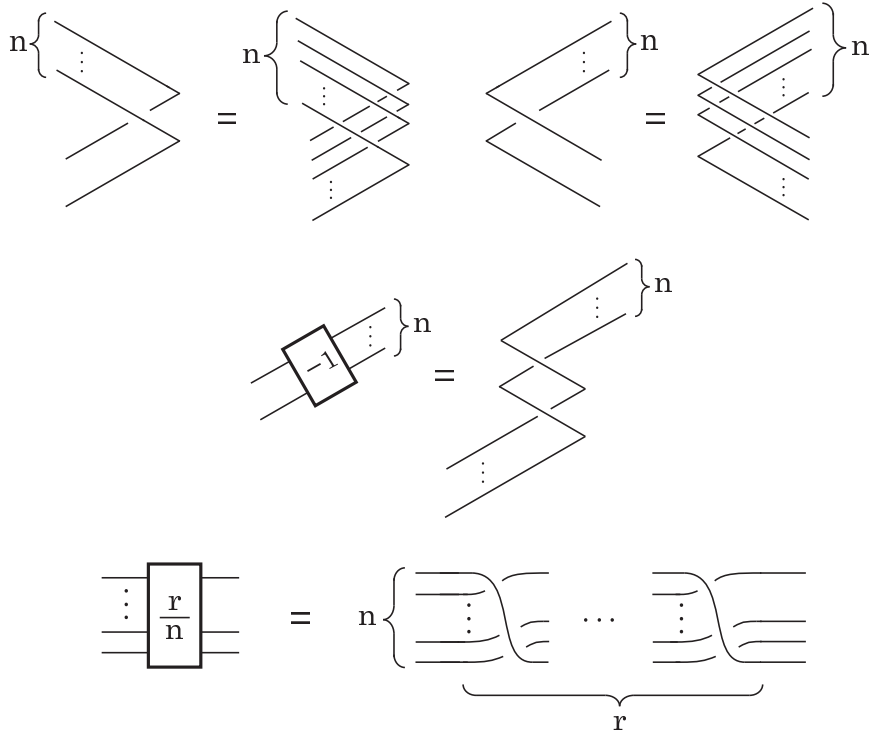}
\caption{Notations on Legendrian versions of twists ($n\geq 1$, $r\geq 1$)}
\label{fig:notation}
\end{center}
\end{figure}

For a knot $K$ in the boundary of a 1-handlebody, a Legendrian knot $\mathcal{K}$ in the boundary is called a \textit{Legendrian representative} of $K$, if $\mathcal{K}$ satisfies the condition below. 
\begin{itemize}
 \item In the case where $K$ is homologically trivial, $\mathcal{K}$ is  smoothly isotopic to $K$. 
 \item In the case where $K$ is homologically non-trivial, $\mathcal{K}$ is smoothly isotopic to $K$ without sliding $\mathcal{K}$ ``over'' any 1-handle. More explicitly, this condition is stated as follows. Consider the dotted circle notation of the 1-handlebody. The condition is that $\mathcal{K}$ is isotopic to $K$ by an isotopy of $S^3$ fixing the disks bounded by the dotted circles. Note that this condition does not allow slidings over the dotted circles. 
 
\end{itemize}
We use this narrow definition to define the maximal Thurston-Bennequin number for a homologically non-trivial knot. 
The \textit{maximal Thurston-Bennequin number} $\overline{tb}(K)$ of a knot $K$ in the boundary of a 1-handlebody is defined to be the maximal value of $tb(\mathcal{K})$ of a Legendrian representative $\mathcal{K}$ of $K$. 
\begin{lemma}For any knot $K$ in the boundary of a 1-handlebody, $\overline{tb}(K)$ is a finite number. 
\end{lemma}
\begin{proof}
In the case where $K$ is null-homologous, this claim immediately follows from the adjunction inequality for general Stein 4-manifolds. In the case where $K$ is homologically non-trivial, we can check this claim as follows. Consider a Legendrian representative $\mathcal{K}$ of $K$ in the boundary of a 1-handlebody. By altering the diagram of $\mathcal{K}$ as shown in Figure~\ref{fig:lem_max}, we obtain a Legendrian knot $\mathcal{K}'$ in $S^3$. Let ${K}'$ denote the smooth knot type of $\mathcal{K}'$.  Due to our narrow definition of a Legendrian representative, we easily see that ${K}'$ is independent of the choice of $\mathcal{K}$. Furthermore, as seen from the diagram, $tb(\mathcal{K}')=tb(\mathcal{K})+\alpha$ for some constant $\alpha$ which depends only on ${K}$. Since $\overline{tb}$ of a knot in $S^3$ is finite, this fact shows that $\overline{tb}(K)$ is also finite. 
\end{proof}
\begin{figure}[h!]
\begin{center}
\includegraphics[width=3.5in]{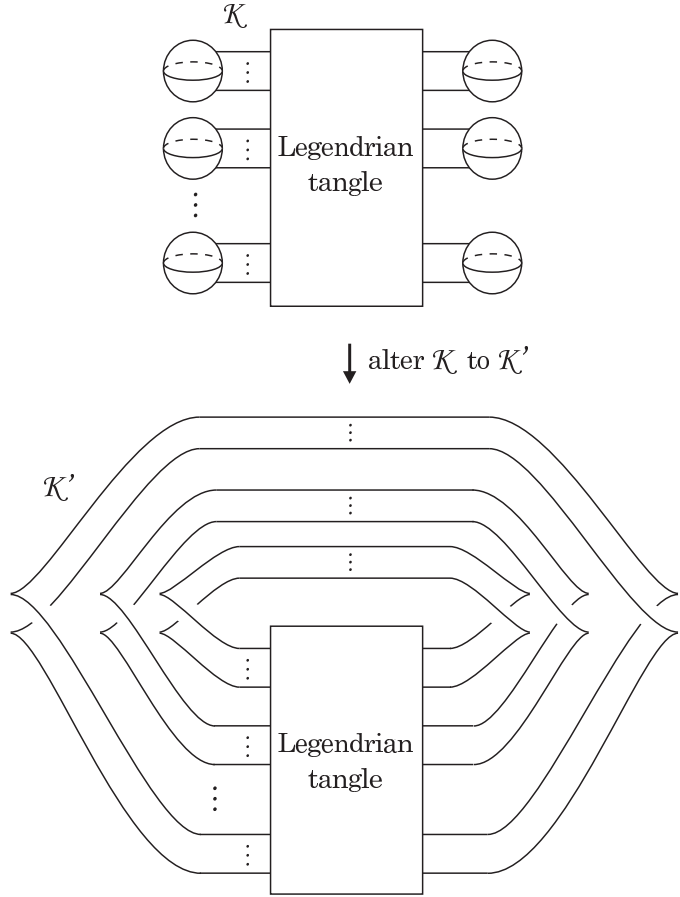}
\caption{Alter a Legendrian knot $\mathcal{K}$ in $\#_n S^1\times S^2$ to a Legendrian knot $\mathcal{K}'$ in $S^3$}
\label{fig:lem_max}
\end{center}
\end{figure}

We remark that, if we change our narrow definition of a Legendrian representative to the natural one, then there are many examples of homologically non-trivial knots with $\overline{tb}=\infty$. 
\section{The method}\label{sec:method}
We give a method for constructing a Legendrian representative of a knot realizing its maximal Thurston-Bennequin number. 
Before the method, we note that a knot $\widetilde{K}$ in the boundary of the sub 1-handlebody $X_1$ of a 2-handlebody $X$ represents a knot in the boundary $\partial X$, since we may regard $\widetilde{K}$ as a knot in $\partial X$ after attaching 2-handles of $X$ to $X_1$. 

Now let $K$ be a knot in $S^3$.  The method proceeds as follows. We do not claim that this procedure is always applicable.
\begin{step}\label{step1}
Find a 2-handlebody $X$ diffeomorphic to $D^4$ such that $K\subset \partial X(\cong S^3)$ is represented by a \textit{good} knot $\widetilde{K}$ in the boundary of the sub 1-handlebody $X_1$ of $X$. 
\end{step}

Here we say that a knot $\widetilde{K}$ in $\partial X_1$ is \textit{good}, if we can draw a Legendrian representative of $\widetilde{K}$ in $\partial X_1$ realizing $\overline{tb}(\widetilde{K})$. For example, torus knots are good knots in this sense. (We define torus knots in $\#_n S^1 \times S^2$ as cables of the unknot in $\#_n S^1 \times S^2$, similarly to the $S^3$ case. Note that an unknot in $\#_n S^1 \times S^2$ is a knot bounding a disk.) Experimentally, the method is effective when we choose a null-homologous knot as $\widetilde{K}$. 

Let $L$ be the link in $\partial X_1$ which consists of the attaching circles of the 2-handles of $X$.

\begin{step}\label{step2}By ignoring the link $L$, first isotope $\widetilde{K}$ to its Legendrian representative $\widetilde{\mathcal{K}}$ in $\partial X_1$ which realizes $\overline{tb}(\widetilde{K})$, and then keep track of the position of the link $L$ in $\partial X_1$ by this isotopy. Next, fixing the position of the Legendrian knot $\widetilde{\mathcal{K}}$, isotope $L$ to its Legendrian representative $\mathcal{L}$ so that the framings of 2-handles of $X$ coincide with the contact $-1$ framing of the Legendrian representative $\mathcal{L}$. Now $X$ is a Stein handlebody. By attaching the 2-handles of $X$ to $X_1$ along $\mathcal{L}$, we regard $\widetilde{\mathcal{K}}$ as a Legendrian knot (denoted by $\mathcal{K}$) in $\partial X\cong S^3$. Here the contact structure on $\partial X$ is the one induced from the Stein structure on $X$. Since $S^3$ has a unique Stein fillable contact structure up to isotopy, $\mathcal{K}$ gives a Legendrian representative of $K$ in the standard tight contact structure on $S^3$. 
\end{step}

We remark a simple sufficient condition that the resulting Legendrian representative $\mathcal{K}$ realizes $\overline{tb}(K)$: if $\widetilde{K}$ bounds a surface of genus $g$ in $X_1$ satisfying $2g-1={tb}(\widetilde{\mathcal{K}})$, then $\overline{tb}(K)={tb}({\mathcal{K}})(={tb}(\widetilde{\mathcal{K}}))$ due to the adjunction inequality. For example, this condition holds if $\widetilde{K}$ is a positive torus knot in $\partial X_1$. Of course, one can also use other upper bounds of $\overline{tb}$ (cf.\ Section~\ref{sec:intro}) to see whether $\mathcal{K}$ realizes $\overline{tb}(K)$. 

\begin{remark}\label{sec:method:remark:construction} $(1)$ (Construction of various examples) One can easily construct various examples of knots for which this method produces Legendrian representatives realizing $\overline{tb}$ as follows. Construct a 2-handlebody $X$ diffeomorphic to $D^4$, and put a null-homologous knot $\widetilde{K}$  on the boundary $\partial X_1$ of the sub 1-handlebody of $X$ satisfying $\overline{tb}(\widetilde{K})=2g_{X_1}(\widetilde{K})-1$. Here $g_{X_1}(\widetilde{K})$ denotes the minimal genus of a smoothly embedded surface in $X_1$ bounded by $\widetilde{K}$. For example, any positive torus knot in $\partial X_1$ satisfies this condition of $\widetilde{K}$. Now let $K$ be a knot in $\partial X\cong S^3$ represented by $\widetilde{K}$. This process corresponds to Step~\ref{step1} of the method. If $\overline{tb}$ of the attaching circle of each 2-handle in $\partial X_1$ is sufficiently larger than the framing of the 2-handle, then one can clearly apply Step~\ref{step2}. By the assumption on $\widetilde{K}$, the resulting Legendrian knot gives a Legendrian representative of $K$ realizing $\overline{tb}(K)$. By using this construction, we can construct many counterexamples to Conjecture~\ref{conjecture}. See the next section.
\smallskip\\ 
$(2)$ (Variant of the method) Although we required that we (can) draw a Legendrian representative of $\widetilde{K}$ realizing $\overline{tb}(\widetilde{K})$ in $\partial X_1$, the method without this condition is still effective for finding a good lower bound of $\overline{tb}(K)$. See Subsection~\ref{subsec:variant} for an example, which also gives counterexamples to Conjecture~\ref{conjecture}. 
\end{remark}

\begin{remark}\label{rem:framing}Regarding the Legendrian representative $\mathcal{K}$ obtained by Step~\ref{step2}, beware that $tb(\mathcal{K})$ may not be the same value as $tb(\widetilde{\mathcal{K}})$. This is because the $0$-framing of  a knot in $S^3$ and that of a knot in the boundary of a 1-handlebody is defined differently. One can easily calculate $tb(\mathcal{K})$ from $tb(\widetilde{\mathcal{K}})$ by checking the difference of $0$-framings induced from $S^3$ and the boundary of the 1-handlebody $X_1$. In particular, if $\widetilde{K}$ is null-homologous in $\partial X_1$, then $tb(\mathcal{K})$ is equal to $tb(\widetilde{\mathcal{K}})$. 
\end{remark}

Since $\mathcal{K}$ is given as a Legendrian knot on the boundary of the Stein handlebody $X$, one might wish to find a representative in the front diagram of $S^3$. 

\begin{step}[Optional]\label{step3}Slide $\mathcal{K}$ over the 2-handles of $X$ so that the resulting knot does not go over any 1-handle and that it is Legendrian preserving $tb(\mathcal{K})$. Then cancel all 1-handles of ${X}$ with 2-handles. The resulting Legendrian knot gives a Legendrian representative of $K$ in the front diagram of $S^3$ realizing $\overline{tb}(K)$. 
\end{step}

This process often yields a very complicated Legendrian diagram. Note that this step is not necessary for determining $\overline{tb}$.

\section{Example}
We demonstrate the method using knots obtained in \cite{Y10}, and we prove Theorem~\ref{intro:thm:reducible}. Let us recall that, for a knot $K$ in $S^3$, the $(p,q)$-cable $C_{p,q}(K)$ of $K$ is defined to be a knot in $S^3$ which is a simple closed curve in the boundary $\partial \nu(K)$ of the tubular neighborhood $\nu(K)$ of $K$ representing the class $p[K']+q[\alpha]$ in $H_1(\partial \nu(K);\mathbb{Z})$. Here $\alpha$ is the positively oriented meridian of $K$, and $K'$ is the $0$-framing of $K$ induced from a Seifert surface of $K$. 

For an integer $m$, let $K_m$ be the ribbon knot in Figure~\ref{fig:ribbon_knot}, where the box denotes the $(m-1)$ right-handed full twists. For an integer $n\geq 2$, let $K_{m,n}$ be the $(n,-1)$ cable of $K_m$. These knots were constructed in \cite{Y10}, and their maximal Thurston-Bennequin numbers were determined for $m\geq 0$ and $n\geq 2$ using rulings and a cabling formula. In this paper, we discuss $\overline{tb}(K_{m,n})$ for $m<0$ using the method introduced in Section~\ref{sec:method}. We remark that the cabling formula of $\overline{tb}$ obtained in \cite{Y10} does not work in this case. 

\begin{figure}[h!]
\begin{center}
\includegraphics[width=1.4in]{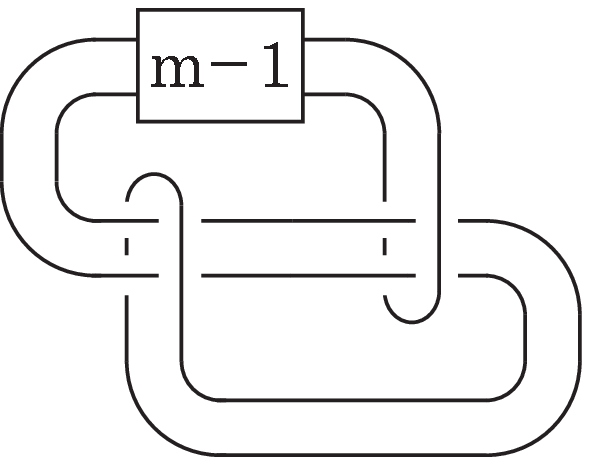}
\caption{$K_m$}
\label{fig:ribbon_knot}
\end{center}
\end{figure}

\subsection{Estimate of $\overline{tb}$ by standard construction}
To see effectiveness of our method, we here estimate $\overline{tb}(K_{m,n})$ by the standard cabling construction of a Legendrian representative (cf.\ Section 5 in \cite{Gol}). Before discussing the $(n,-1)$ cable $K_{m,n}$, we discuss its companion $K_m$. Since $K_m$ bounds a disk in $D^4$, the adjunction inequality shows $\overline{tb}(K_m)\leq -1$. On the other hand, for $m\leq -1$, we can easily check that $K_m$ is isotopic to the Legendrian knot $\mathcal{K}_m$ with $tb=-1$ shown in Figure~\ref{fig:Legendrian_basic}. Therefore, this Legendrian representative of $K_m$ realizes $\overline{tb}(K_m)=-1$. 

We draw $n$ copies of the front diagram of $\mathcal{K}_m$ each of which is slightly shifted to the vertical direction. Inserting $\frac{n-1}{n}$ right handed full twists to the resulting diagram appropriately, we obtain the Legendrian representative of $K_{m,n}$ in Figure~\ref{fig:cable}. Calculating the number of left cusps and the writhe, one can easily check that $tb$ of this representative is $-2n+1$. Thus the standard construction merely shows $\overline{tb}(K_{m,n})\geq -2n+1$. It seems difficult to realize a larger Thurston-Bennequin number by modifying this representative. 

\begin{figure}[h!]
\begin{center}
\includegraphics[width=2.7in]{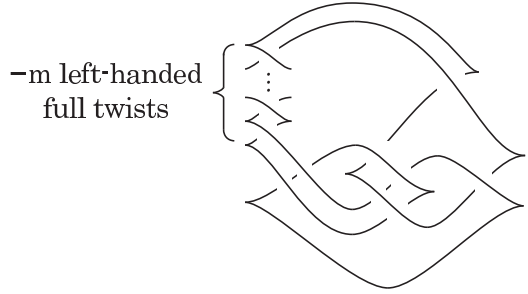}
\caption{A Legendrian representative of $K_m$ with $tb=-1$ $(m\leq -1)$}
\label{fig:Legendrian_basic}
\end{center}
\end{figure}

\begin{figure}[h!]
\begin{center}
\includegraphics[width=3.1in]{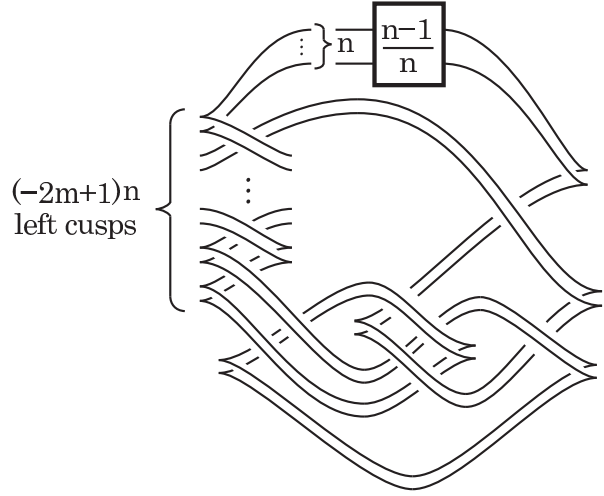}
\caption{A Legendrian representative of $K_{m,n}$ with $tb=-2n+1$ $(m\leq -1)$}
\label{fig:cable}
\end{center}
\end{figure}

\subsection{Step~\ref{step1}} Now we apply our method to $K_{m,n}$. We give necessary definitions to proceed Step~\ref{step1} of the method. For an integer $m$, let $Z^{(m)}$ be the 4-manifold shown in Figure~\ref{fig:Z_m}. Since the 2-handle goes over the 1-handle geometrically once after isotopy, $Z^{(m)}$ is diffeomorphic to $D^4$. We identify the boundary $\partial Z^{(m)}$ with $S^3$ via this diffeomorphism. Let $\widetilde{K}_m$ and  $\widetilde{K}_{m,n}$ be the unframed knots in $\partial Z^{(m)}=S^3$ given by Figures~\ref{fig:tilde_K_m} and \ref{fig:tilde_K_mn}, respectively. 
\begin{figure}[h!]
\begin{center}
\includegraphics[width=1.9in]{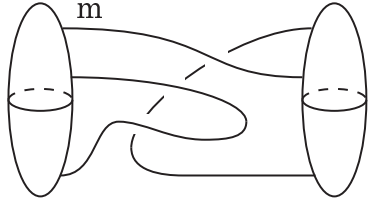}
\caption{The handlebody $Z^{(m)}$ which is diffeomorphic to $D^4$}
\label{fig:Z_m}
\end{center}
\end{figure}
\begin{figure}[h!]
\begin{center}
\includegraphics[width=1.9in]{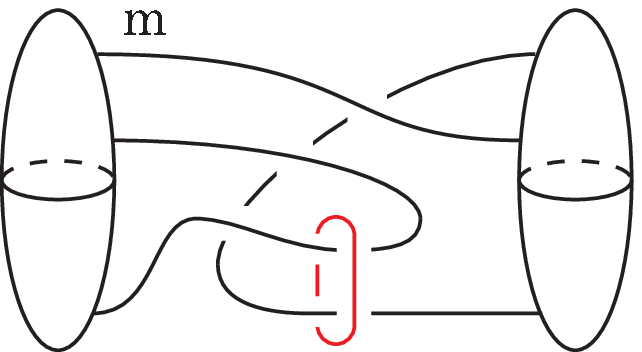}
\caption{The knot $\widetilde{K}_m$ in $\partial Z^{(m)}$}
\label{fig:tilde_K_m}
\end{center}
\end{figure}
\begin{figure}[h!]
\begin{center}
\includegraphics[width=2.5in]{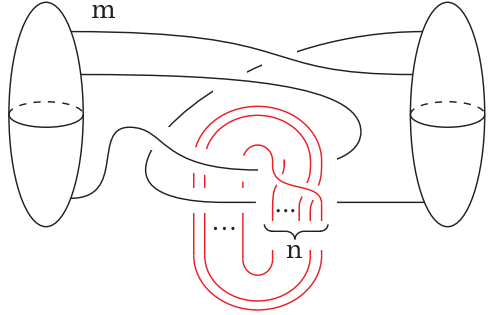}
\caption{The knot $\widetilde{K}_{m,n}$ in $\partial Z^{(m)}$}
\label{fig:tilde_K_mn}
\end{center}
\end{figure}

To apply Step~\ref{step1}, we show the lemma below. 
\begin{lemma}For integers $m,n$ with $n\geq 2$, the knots $K_{m}$ and $K_{m,n}$ are isotopic to the knots $\widetilde{K}_m$ and $\widetilde{K}_{m,n}$ in $\partial Z^{(m)}$, respectively. 
\end{lemma}
\begin{proof}By Figure~\ref{fig:cancel_knot}, we see that $\widetilde{K}_{m}$ is isotopic to $K_m$. 
Figures~\ref{fig:tilde_K_m} and \ref{fig:tilde_K_mn} and the definition of a cable knot show that $\widetilde{K}_{m,n}$ is the $(n,-1)$ cable of $\widetilde{K}_{m}$. Hence $\widetilde{K}_{m,n}$ is isotopic to ${K}_{m,n}$. 
\end{proof}
\begin{figure}[h!]
\begin{center}
\includegraphics[width=4.2in]{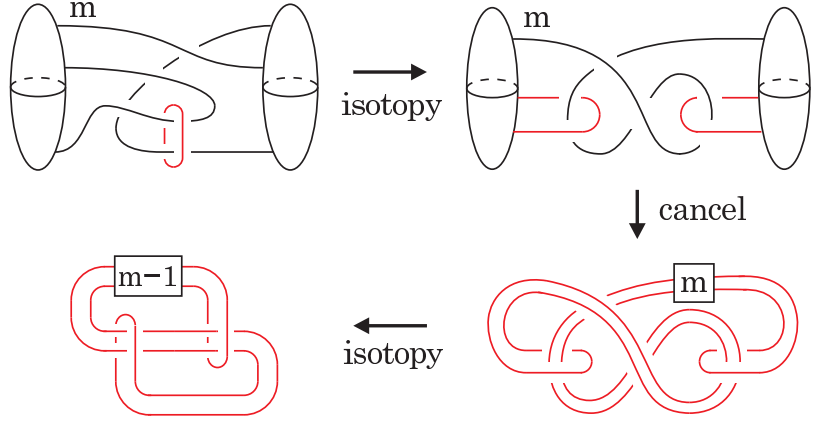}
\caption{Diagrams of the knot $K_m$ in $S^3$}
\label{fig:cancel_knot}
\end{center}
\end{figure}

We regard $\widetilde{K}_{m,n}$ as a knot in the boundary of the sub 1-handlebody $Z^{(m)}_1$ of $Z^{(m)}$. Then $\widetilde{K}_{m,n}$ is clearly an unknot in the boundary $\partial Z^{(m)}_1$, and we know a Legendrian representative of an unknot realizing $\overline{tb}$. Therefore we finished Step~\ref{step1}.

\subsection{Step~\ref{step2}}Next we apply Step~\ref{step2}. We note that, if the framing $m$ of the 2-handle of $Z^{(m)}$ is a sufficiently large negative number,  then we can obviously achieve this step. We first isotope $\widetilde{K}_{m,n}$ to its Legendrian representative realizing $\overline{tb}$ in $\partial Z^{(m)}_1$, and then we keep track of the 2-handle of $Z^{(m)}$ as shown in Figure~\ref{fig:local_unknot}. By putting the 2-handle into a Legendrian position, we obtain the Legendrian representative of $\widetilde{K}_{m,n}$ in $\partial Z^{(m)}_1$ shown in Figure~\ref{fig:Legendrian_1-h} for $m\leq -4n+3$. Note that $Z^{(m)}$ is now a Stein handlebody and that $tb(\widetilde{K}_{m,n})=-1$. The Legendrian representative of $\widetilde{K}_{m,n}$ thus gives a Legendrian representative of ${K}_{m,n}$ in the boundary $\partial Z^{(m)}$ of the Stein handlebody, since $\widetilde{K}_{m,n}$ represents $K_{m,n}$ in $\partial Z^{(m)}$. 

Since $\widetilde{K}_{m,n}$ bounds a disk in $Z^{(m)}_1$, the adjunction inequality shows $\overline{tb}(K_{m,n})\leq -1$. Therefore the Legendrian representative of $K_{m,n}$ in Figure~\ref{fig:Legendrian_1-h} realizes $\overline{tb}=-1$. Note that this value of $\overline{tb}$ is equal to the one induced from the front diagram of $S^3$ (see Remark~\ref{rem:framing}). This completes Step~\ref{step2}, and the proposition below follows. 
\begin{proposition}\label{prop:tb}$\overline{tb}(K_{m,n})=-1$ for $n\geq 2$ and $m\leq -4n+3$. 
\end{proposition}

\begin{figure}[h!]
\begin{center}
\includegraphics[width=4.5in]{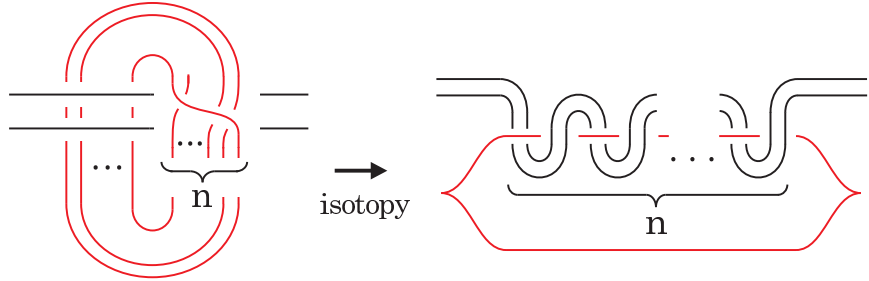}
\caption{local isotopy}
\label{fig:local_unknot}
\end{center}
\end{figure}

\begin{figure}[h!]
\begin{center}
\includegraphics[width=4.5in]{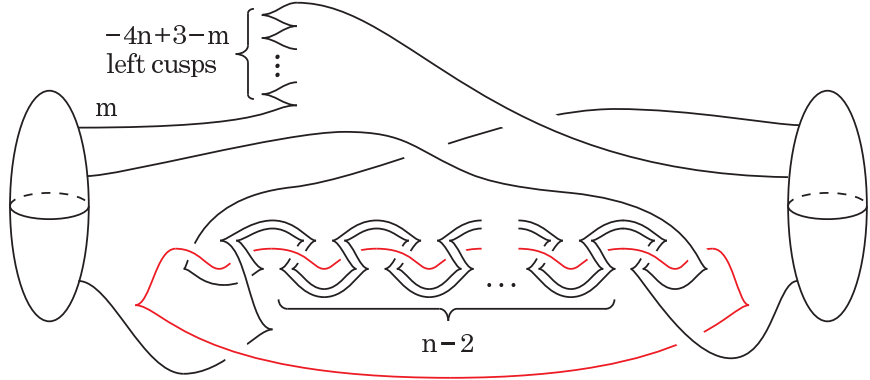}
\caption{Legendrian representative of $K_{m,n}$ with $tb=-1$ in the Stein handlebody diagram of $Z^{(m)}$ $(m\leq -4n+3)$}
\label{fig:Legendrian_1-h}
\end{center}
\end{figure}

Now we can easily prove Theorem~\ref{intro:thm:reducible}. 
\begin{proof}[Proof of Theorem~\ref{intro:thm:reducible}]For integers $n$ and $k$, let $X_{n,k}$ and $Y_{n,k}$ be the 4-manifolds in Figures~\ref{fig:X_n,k} and \ref{fig:Y_n,k}, respectively. Note that $\partial X_{n,k}$ is a homology 3-sphere, since $X_{n,k}$ is contractible. By \cite{Y10}, the 4-manifold represented by $K_{m,n}$ with $-n$-framing is diffeomorphic to $Y_{n,m+4n}$ for $n\geq 2$. The $-n$-surgery along $K_{m,n}$ thus yields the 3-manifold $\partial Y_{n,m+4n}$, which is clearly diffeomorphic to the connected sum $\partial{X_{n,m+4n}}\#L(n,1)$. Here $L(n,1)$ denotes the lens space given by $-n$-surgery along the unknot, following the convention in contact topology. The knot $K_{m,n}$ thus yields a reducible 3-manifold by $-n$-surgery for $n\geq 2$, since $\partial{X_{n,m+4n}}$ is not diffeomorphic to $S^3$ (\cite{Y10}). Here recall that the $r$-surgery along the $(p,q)$-cable of a non-trivial knot in $S^3$ yields a reducible 3-manifold if and only if $r=pq$ (Theorem~3 in \cite{GL}). Thus $K_{m,n}$ is not isotopic to $K_{m,n'}$ if $n\neq n'$. Therefore, by Proposition~\ref{prop:tb}, the infinite family of knots $\{K_{m,n}\mid n\geq 2, \; m\leq -4n+3\}$ satisfies the desired conditions. 
\end{proof}

\begin{figure}[h!]
\begin{center}
\includegraphics[width=1.2in]{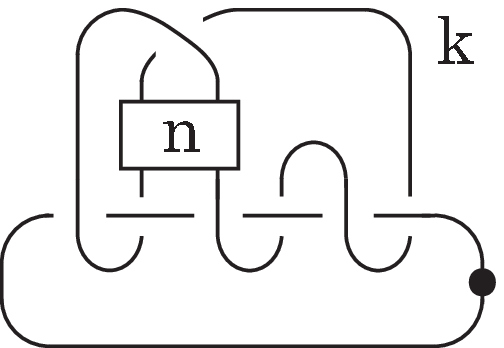}
\caption{$X_{n,k}$}
\label{fig:X_n,k}
\end{center}
\end{figure}
\begin{figure}[h!]
\begin{center}
\includegraphics[width=1.7in]{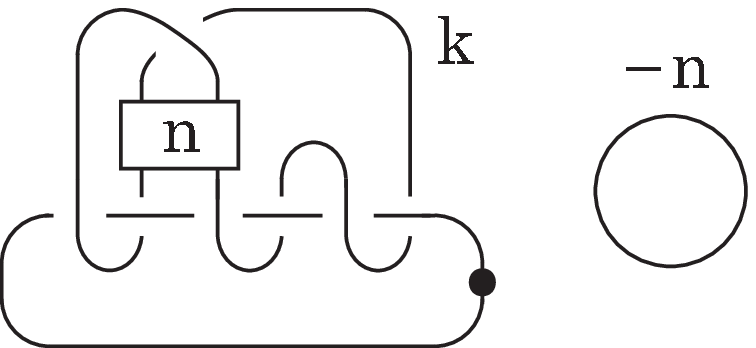}
\caption{$Y_{n,k}$}
\label{fig:Y_n,k}
\end{center}
\end{figure}
\begin{remark}We can construct many other counterexamples by using the construction in Remark~\ref{sec:method:remark:construction}. Indeed, if we construct $X$ and $\widetilde{K}$ so that $\widetilde{K}$ is unknot and that $K$ is the $(n,-1)$-cable of a non-trivial knot in $S^3$, then $\overline{tb}(K)=-1$, and $K$ yields a reducible 3-manifold by $-n$-surgery, giving a counterexample to Conjecture~\ref{conjecture}. 
\end{remark}
\subsection{Step~\ref{step3}} Finally we apply Step~\ref{step3} to obtain a Legendrian representative of $K_{m,n}$ realizing $\overline{tb}$ in the front diagram of $S^3$. 
We first apply local isotopies in Figure~\ref{fig:local_cancel} to the Stein handlebody diagram of $Z^{(m)}$ and the knot $K_{m,n}$. Note that these isotopies preserve $tb$ of $K_{m,n}$ and the 2-handle. To simplify the diagram, we use tangles $A$ and $B$ defined in Figure~\ref{fig:AB}. The resulting diagram  of $K_{m,n}$ and $Z^{(m)}$ is shown in the first diagram of Figure~\ref{fig:complicated}. Now we can easily isotope the 2-handle and $K_{m,n}$ so that the 2-handle goes over the 1-handle geometrically once and that $tb$ of these knots do not change. The resulting diagram is given in the second diagram of Figure~\ref{fig:complicated} (This isotopy can be easily seen by ignoring $K_{m,n}$ and tangles $A$ and $B$.).  

\begin{figure}[h!]
\begin{center}
\includegraphics[width=3.2in]{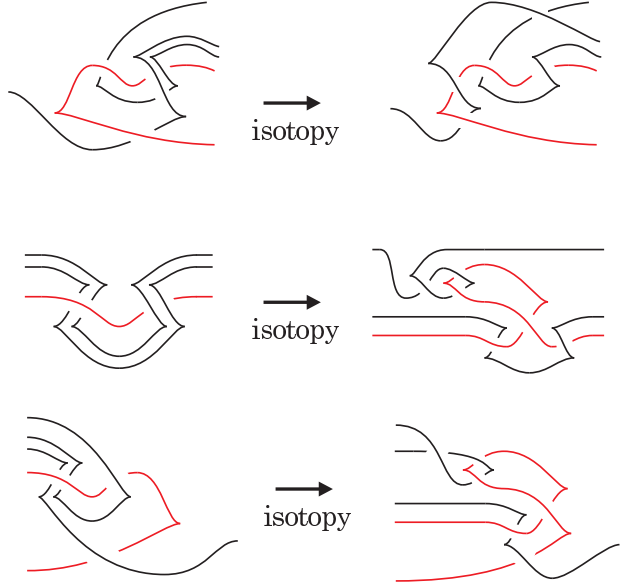}
\caption{Isotopies fixing the end points}
\label{fig:local_cancel}
\end{center}
\end{figure}

\begin{figure}[h!]
\begin{center}
\includegraphics[width=4.0in]{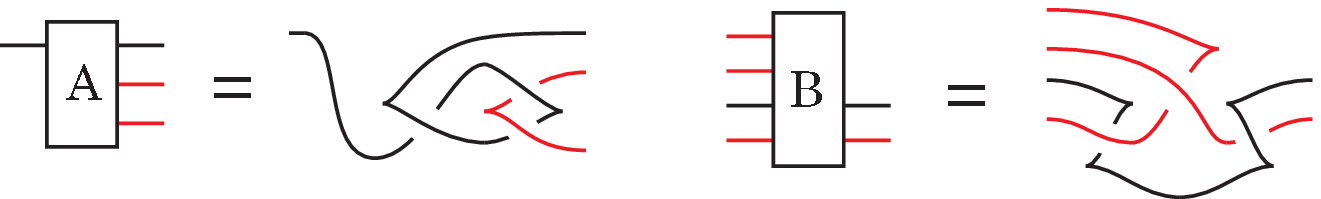}
\caption{Definition of tangles $A$ and $B$}
\label{fig:AB}
\end{center}
\end{figure}

\begin{figure}[h!]
\begin{center}
\includegraphics[width=4.5in]{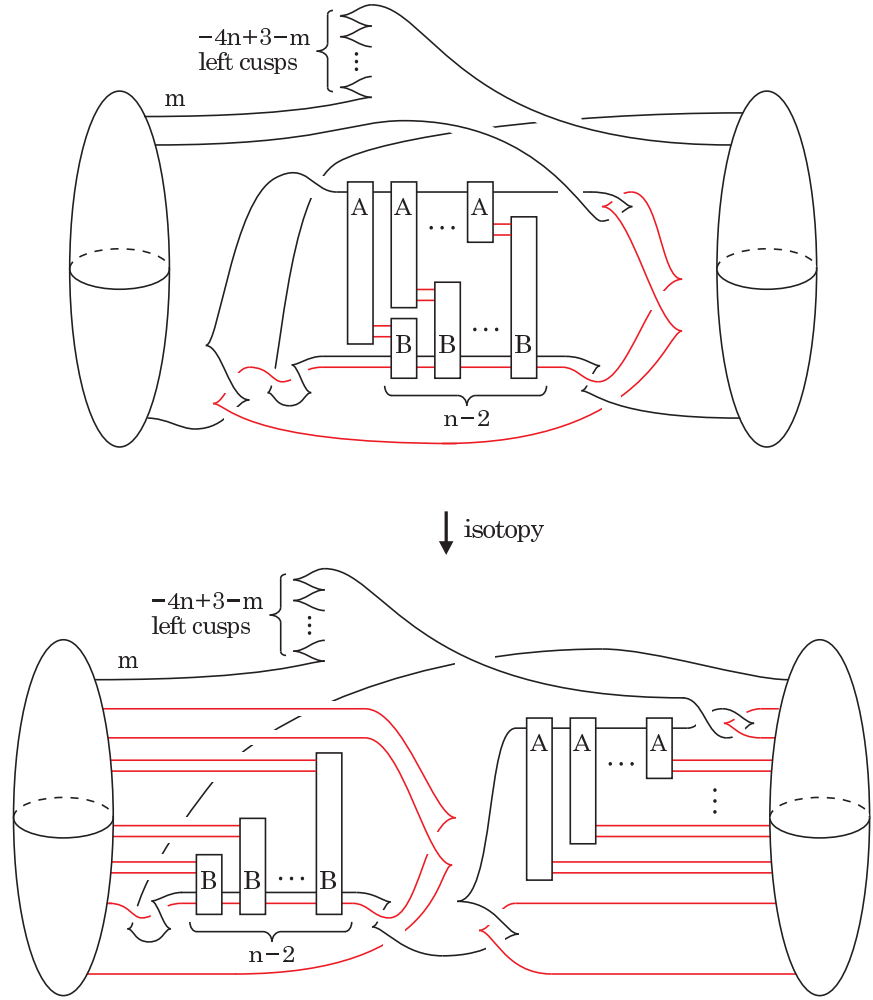}
\caption{Legendrian representatives of $K_{m,n}$ in Stein handlebody diagrams of $Z^{(m)}$ ($n\geq 2$ and $m\leq -4n+3$)}
\label{fig:complicated}
\end{center}
\end{figure}

Next we slide $K_{m,n}$ over the 2-handle so that $K_{m,n}$ does not go over the 1-handle (after suitable isotopy) and that $tb(K_{m,n})$ does not change. More specifically, we slide $K_{m,n}$ at the right most part of Figure~\ref{fig:complicated} as shown in Figure~\ref{fig:slide_local}, where the framings of the 2-handle are the contact $-1$ framings. We can easily check that this operation preserves $tb(K_{m,n})$ by counting the numbers of positive crossings, negative crossings and left cusps. Clearly we can isotope the resulting $K_{m,n}$ preserving $tb$ so that it does not go over the 1-handle.  Now we can cancel (erase) the canceling pair of 1- and 2-handles. The resulting diagram is given in Figure~\ref{fig:knot_complicated}, where we use the tangles $A_n, B_n$ defined in Figure~\ref{fig:ABn}. Therefore, this diagram gives a Legendrian representative of $K_{m,n}$ realizing $\overline{tb}=-1$ in the front diagram of $S^3$. We thus completed Step~\ref{step3}. 

\begin{figure}[h!]
\begin{center}
\includegraphics[width=4.0in]{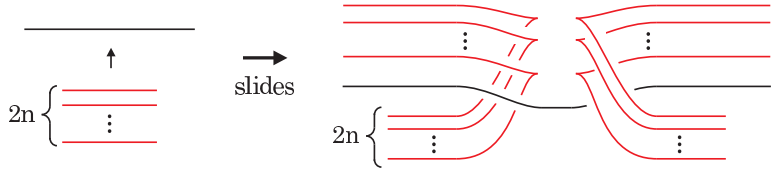}
\caption{sliding of $K_{m,n}$ over the 2-handle}
\label{fig:slide_local}
\end{center}
\end{figure}

\begin{figure}[h!]
\begin{center}
\includegraphics[width=3.2in]{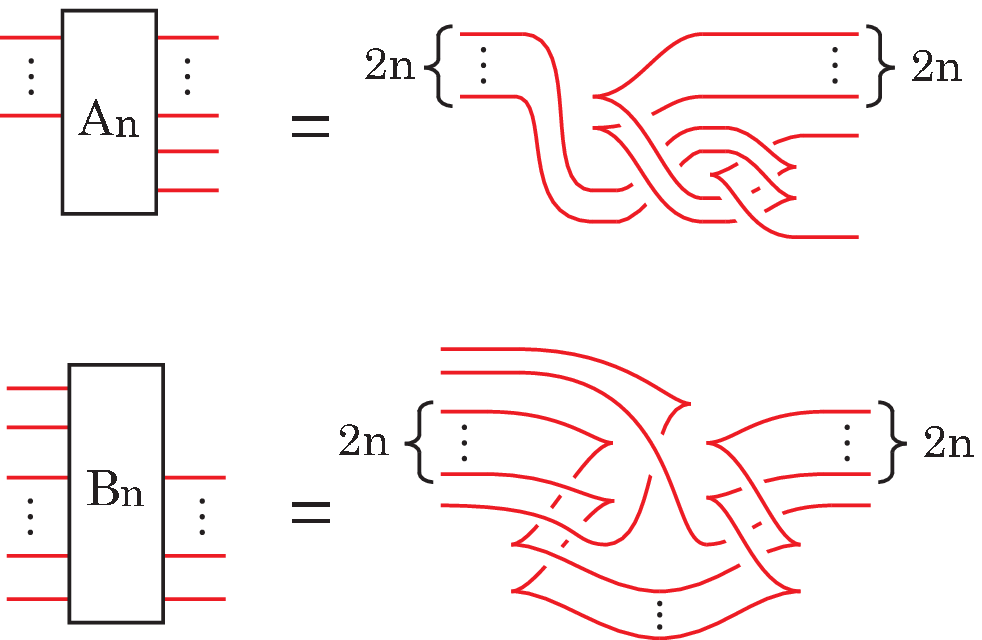}
\caption{Definition of tangles $A_n$ and $B_n$}
\label{fig:ABn}
\end{center}
\end{figure}

\begin{figure}[h!]
\begin{center}
\includegraphics[width=4.8in]{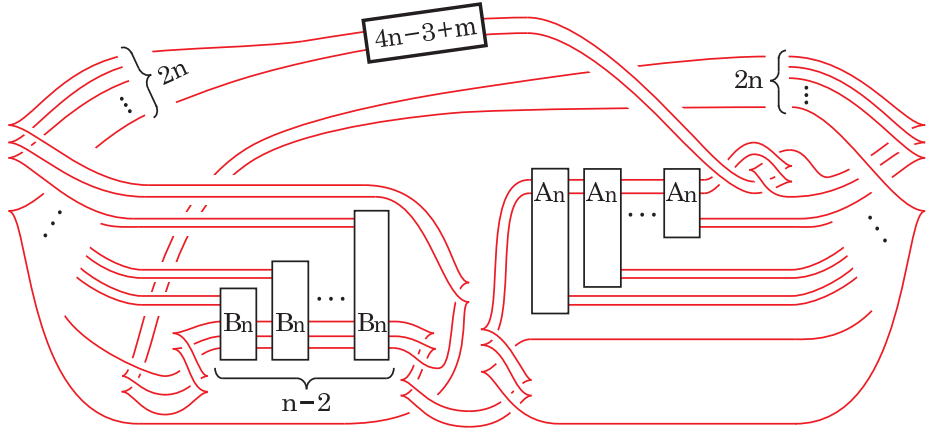}
\caption{Legendrian representative of $K_{m,n}$ realizing $\overline{tb}=-1$ ($n\geq 2$ and $m\leq -4n+3$)}
\label{fig:knot_complicated}
\end{center}
\end{figure}

\subsection{Variant of the method}\label{subsec:variant}
As we mentioned in Remark~\ref{sec:method:remark:construction}, our method is also effective for finding a good lower bound of $\overline{tb}$ by minor modification. We demonstrate this using the knot $K_{m,n}$ with $n\geq 2$ and $m\leq -2n-1$. 

Recall that $K_{m,n}$ is isotopic to the unframed knot $\widetilde{K}_{m,n}$ in the boundary of the handlebody $Z^{(m)}$ shown in Figure~\ref{fig:tilde_K_mn}. We regard $\widetilde{K}_{m,n}$ as a knot in the boundary of the sub 1-handlebody $Z^{(m)}_1$ of $Z^{(m)}$.  Since $\widetilde{K}_{m,n}$ is an unknot in $\partial Z^{(m)}_1$, we can isotope $\widetilde{K}_{m,n}$ to its Legendrian representative with $tb=-n+1$. We then isotope the 2-handle of $Z^{(m)}$ to its Legendrian representative fixing the representative of $\widetilde{K}_{m,n}$. The resulting diagram of $\widetilde{K}_{m,n}$ and $Z^{(m)}$ is shown in Figure~\ref{fig:variant}. Clearly this diagram gives a Stein handle decomposition of $Z^{(m)}$, and thus the Legendrian representative of $\widetilde{K}_{m,n}$ gives a Legendrian representative of $K_{m,n}$ with $tb=-n+1$ in the Stein fillable contact structure on $\partial Z^{(m)}\cong S^3$. Hence the proposition below follows.

\begin{figure}[h!]
\begin{center}
\includegraphics[width=3.5in]{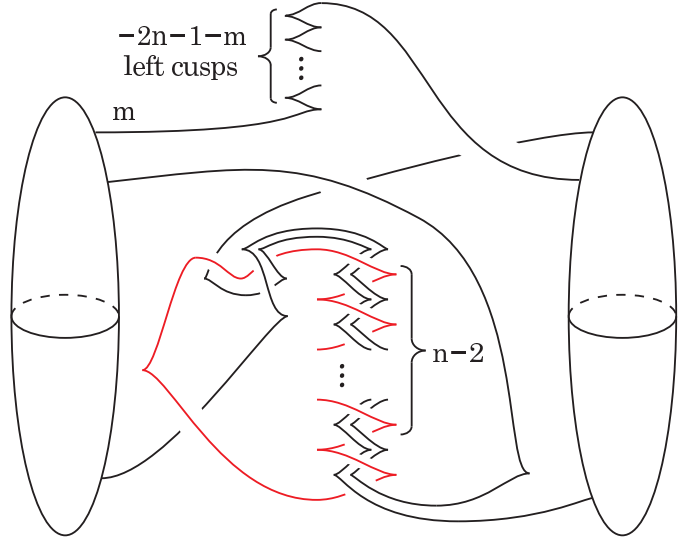}
\caption{Legendrian representative of $K_{m,n}$ with ${tb}=-n+1$ in the Stein handlebody diagram of $Z^{(m)}$ ($n\geq 2$ and $m\leq -2n-1$)}
\label{fig:variant}
\end{center}
\end{figure}

\begin{proposition}\label{prop:variant}$\overline{tb}(K_{m,n})\geq -n+1$ for $n\geq 2$ and $m\leq -2n-1$. 
\end{proposition}
It seems difficult to obtain this estimate without using a handlebody diagram of $D^4$. We remark that we can also draw a Legendrian representative with $tb=-n+1$ in the front diagram of $S^3$, similarly to Step~\ref{step3}. 

\begin{remark}In \cite{Y10}, 
we discussed the following problem. ``Assume that a framed knot in $S^3$ represents a 4-manifold admitting a Stein structure. Is the framing less than the maximal Thurston-Bennequin number of the knot?'' 
Since we proved in \cite{Y10} that the 4-manifold represented by $-n$-framed $K_{m,n}$ admits a Stein structure for $n\geq 2$ and $m\leq -2n-1$, it is natural to ask if the framing $-n$ is less than $\overline{tb}(K_{m,n})$. (We showed the existence of a Stein structure by checking that this 4-manifold is diffeomorphic to the boundary connected sum of two compact Stein 4-manifolds.) The above proposition tells that the framing is indeed less than $\overline{tb}(K_{m,n})$, giving a supporting evidence for the above problem. 
\end{remark}

Characterizing an unknot is a natural question in knot theory, and various characterizations are known. Here we propose the following question as a potential characterization given by maximal Thurston-Bennequin numbers. Recall that $C_{p,q}(K)$ denotes the $(p,q)$-cable of a knot $K$ in $S^3$. 

\begin{question}If a knot $K$ in $S^3$ satisfies $\overline{tb}(C_{p,-1}(K))=-1$ for any positive integer $p$, is $K$ the unknot?
\end{question}

We remark that, for each positive integer $N$, Proposition~\ref{prop:tb} implies the existence of a non-trivial knot $K$ satisfying $\overline{tb}(C_{p,-1}(K))=-1$ for any positive integer $p\leq N$. 



\end{document}